\documentclass[12pt]{amsart}
\usepackage{amsmath,ifthen,amsthm,srcltx,amsopn,amssymb,amsfonts}
\usepackage{amscd}
\usepackage[english]{babel}

\usepackage[all,cmtip]{xy}
\usepackage[T1]{fontenc}
\usepackage[utf8x]{inputenc}
\usepackage{hyperref}
\usepackage{psfrag}
\usepackage{delarray,graphicx}
\usepackage{comment}
\usepackage[all,cmtip]{xy}

\usepackage{graphicx}

\usepackage{amssymb,latexsym}
\usepackage{color}

\usepackage{eucal}

\title[Representations of $3$-manifolds groups in $\PGL(n,\C)$]{Representations of $3$-manifolds groups in $\PGL(n,\C)$ and their restriction to the boundary}

\author{Antonin Guilloux} 
\address{Institut de Math\'ematiques de Jussieu \\
Unit\'e Mixte de Recherche 7586 du CNRS \\
Universit\'e Pierre et Marie Curie \\
4, place Jussieu 75252 Paris Cedex 05, France \\}
\email{aguillou@math.jussieu.fr}
\urladdr{http://people.math.jussieu.fr/~aguillou}

\begin{document}

\renewcommand{\H}{\mathbf{H}}
\newcommand{\Z}{\mathbf{Z}}
\newcommand{\R}{\mathbf{R}}
\newcommand{\C}{\mathbf{C}}
\newcommand{\PGL}{\mathrm{PGL}}
\newcommand{\SL}{\mathrm{SL}}
\newcommand{\Hom}{\mathrm{Hom}}
\newcommand{\periph}{\mathrm{periph}}
\newcommand{\geom}{\mathrm{geom}}
\newcommand{\Hol}{\mathrm{Hol}}
\newcommand\PU{\textrm{PU}}
\newcommand\Fl{\mathcal{F}l}
\newcommand\AFl{\mathcal{AF}l}
\renewcommand\P{\mathbb{P}}

\newtheorem{proposition}{Proposition}

\newtheorem{theorem}{Theorem}
\newtheorem{lemma}{Lemma}
\newtheorem{fact}{Fact}
\newtheorem{cor}{Corollary}

\theoremstyle{definition}
\newtheorem{definition}{Definition}
\newtheorem{remark}{Remark}

\begin{abstract}  
Let $M$ be a cusped $3$-manifold -- e.g. a knot complement -- and note $\partial 
M$ the collection of its peripheral tori. Thurston \cite{thurston-gt3m} gave a 
combinatorial way to produce hyperbolic structures via triangulation and the 
so-called \emph{gluing equations}. This gives coordinates on the space of 
representations 
of $\pi_1(M)$ to $\PGL(2,\C)$.

In their paper \cite{NeumannZagier}, Neumann and Zagier showed how this 
coordinates are adapted to describe this space of representations as a 
lagragian subvariety lying inside a space equipped with a $2$-form -- now 
called Neumann-Zagier symplectic space. And they related this $2$-form with a 
natural symplectic form on the space of representations of $\pi(\partial M)$ 
to $\PGL(2,\C)$: the Weil-Petersson form.

Subsequent works of Neumann \cite{Neumann} and Kabaya \cite{kabaya} extended 
the scope of the previous works. And, more recently, there has been 
generalizations of this strategy for representations to $\PGL(3,\C)$ \cite{BFG} 
or $\PGL(n,\C)$ \cite{GaroufalidisThurstonZickert, GaroufalidisGoernerZickert, DGG}. 
Unfortunately, in the $\PGL(n,\C)$-case, the program of Neumann-Zagier has not 
been fulfilled: indeed the second part (the link with the Weil-Petersson form) 
was not achieved. At the very end of the process of writing this paper, a paper of Garoufalidis and Zickert appeared on the ArXiv \cite{GZ}. Their result are very similar, though the point of view is slightly different. 

We exhibit in this note such a symplectic morphism. It is a direct 
generalization of the work \cite{BFG}, with the key input of the 
parametrization given in \cite{DGG}.
\end{abstract}

\maketitle

\section{Introduction}

Let $M$ be the $8$-knot complement. Thurston \cite{thurston-gt3m} explained the following program to construct 
its hyperbolic structure:
\begin{enumerate}
\item Triangulate $M$, here thanks to the Riley's triangulation.
\item Give a set of parameters to each tetrahedra, here cross-ratios, that 
describe their hyperbolic structure.
\item Glue back the tetrahedra, imposing the \emph{gluing equations}. Those 
insure that the edge will not become singular.
\item Add a polynomial condition specifying that the structure is complete, by 
forcing the peripheral holonomy to be parabolic.
\end{enumerate}
Hence the hyperbolic structure is described by the solution to a polynomial 
system. Moreover, relaxing the last condition, this parametrize a (Zariski-)open subset of a decorated 
version of the character variety:
$$\chi_2(M):=\Hom(\pi_1(M),\PGL(2,\C))//\PGL(2,\C).$$
This approach has proven very efficient and is followed in the computer program SnapPy to construct 
hyperbolic structures on ideally triangulated $3$-manifolds.

This program was further developed by Neumann and Zagier in 
\cite{NeumannZagier}. By a careful analysis of items 2 and 3, they showed that 
there is a $\C$-vector space (denoted $\ker(\beta^*)\subset J$ in \cite{Neumann}) carrying an 
antisymmetric bilinear form $\omega$ such that 
\begin{itemize}
\item the character $\chi_2(M)$, through the parameters, is seen as a subvariety of 
$\exp(\ker(\beta^*))$ tangent to the kernel of the $2$-form $\omega$\footnote{More precisely, it is the decorated character variety.}.
\item the symplectic quotient $\mathcal H(J)$ of $\ker(\beta^*)$ (the so-called Neumann-Zagier symplectic  
space) is 
isomorphic to the cohomology group $H^1(\partial M,\C)$ with its Goldman-Weil-
Peterson symplectic form ($\partial M$ denotes the peripheral torus).
\end{itemize}
This presentation uses the more precise version given by Neumann 
\cite{Neumann}. This construction allows to understand the volume of 
the representations near the holonomy of the hyperbolic structure 
\cite{NeumannZagier}. It has been used to give a proof of the local 
rigidity of the holonomy of the hyperbolic structure \cite{Choi}. Kabaya \cite{kabaya}
investigated the case of $M$ being a compact hyperbolic manifold with higher genus boundary.

More recently, several new works revisited Neumann-Zagier strategy and generalized it to understand the character variety:
$$\chi_n(M):=\Hom(\pi_1(M),\PGL(n,\C))//\PGL(n,\C).$$
The reasons of this new interest seems to emanate from two very different fields. First,
from a geometric point of view: the construction of representations $\pi_1(M)\to \mathrm{PU}(2,1)$,
following the initial strategy of Thurston, has been undertaken by Falbel \cite{falbel} in order to 
investigate the possibility for $M$ to carry a CR-spherical structure. Using Neumann-Zagier approach, 
Bergeron, Falbel and the author \cite{BFG} gave a description of $\chi_3(M)$ similar to the one of $\chi_2
(M)$ described above. This leads to a local rigidity result \cite{BFGKR} and actual computations (for $n=3$) \cite
{FKR}. Those, in turn, leads to construction of geometric structures \cite{deraux-falbel}. Another 
approach is via physical mathematics. I must confess my ignorance and refer to Dimofte and Garoufalidis 
\cite{DimofteGaroufalidis} for a presentation. This motivated the works of Garoufalidis, Goerner, 
Thurston and Zickert \cite{Zickert,GaroufalidisThurstonZickert,GaroufalidisGoernerZickert}. They proposed 
a set of parameters for the case $\PGL(n,\C)$, and generalized partially Neumann-Zagier results for their 
setting. This also leads to actual computations (mainly when $n=3$) by the second named author. Dimofte, 
Gabella and Goncharov \cite{DGG} also analyzed the problem for $\PGL(n,\C)$ from this point of view, 
giving a systematic 
account of a set of coordinates, together with the announcement that they are 
able to fulfill the Neumann-Zagier strategy. Unfortunately all the proofs are not given in their paper. As 
mentioned in the abstract, by the very end of the writing of this paper, 
Garoufalidis and Zickert \cite{GZ} published another version of this work. 
Their result and the one discussed in this paper are very similar. However, in my opinion, from a geometrical viewpoint, the approach here allows a better understanding\footnote{I think that the point raised in their remark 2.12 is answered here.}. As an application of our approach, this gives a variational formula for the volume of a representation, as thoroughly discussed in \cite{DGG}. Here we present another, more geometric, application: we prove the local rigidity result generalizing \cite{Choi,BFGKR}.

This paper links the work of \cite{DGG} with \cite{BFG} to complete Neumann-Zagier program in the case of 
$\PGL(n,\C)$. My feeling is that the coordinates given in 
\cite{DGG} are very well adapted to understand of the "lagrangian part" of the strategy of Neumann-Zagier -- i.e. describe the analog of the vector space $\ker(\beta^*)\subset J$ with its form $\omega$ such that $\chi_n(M)$ 
is 
tangent to its kernel in $\exp(\ker(\beta^*))$ -- and define the volume of those representations. But, in order to 
understand the "symplectic isomorphism part", a 
direct generalization of \cite{BFG} seems suitable. 

After this rather long introduction, let me warn the reader that this paper heavily relies on three 
sources: 
\begin{itemize}
\item Fock and Goncharov combinatorics described in \cite{FG},
\item Dimofte, Gabella and Goncharov work in \cite{DGG},
\item Bergeron, Falbel and G. work in \cite{BFG} (and through it to the original Neumann-Zagier strategy 
\cite{NeumannZagier,Neumann}.
\end{itemize}
Those works are not easily resumed. So I rather choosed to give precise references to them. This makes 
this paper absolutely not self-contained. I plan to write later on a more thorough presentation.

\section{Triangulation, flags, affine flags and their configurations}

\subsection{Triangulated manifold}

We will consider in this paper triangles and tetrahedra. Those will always be oriented: an orientation is 
an ordering of the vertices up to even permutations. Note that the faces of a tetrahedron inherits an 
orientation.

An abstract triangulation is defined as a pair $\mathcal{T}=( (T_{\nu}
)_{\nu = 1 , \ldots , N} , \Phi)$ where $(T_{\nu} )_{\nu = 1 ,\ldots , N}$ is a finite family of 
tetrahedra and $\Phi$ is a matching of the faces of the $T_{\nu}$'s reversing the orientation.  
For any tetrahedron $T$, we define $\mathrm{Trunc} (T)$ as
the tetrahedron truncated at each vertex. The space obtained from $\mathrm{Trunc} (T_{\mu})$  after 
matching the faces will be denoted by $K_\mathcal {T}$.  

A \emph{triangulation} of an oriented compact $3$-manifold $M$ with boundary is an abstract triangulation 
$\mathcal T$ together with an oriented homeomorphism $M\simeq K_\mathcal{T}$.

Remark that a knot complement is homeomorphic to the interior of such a triangulated manifold 
\cite[Section 1.2]{BFG}. And a theorem of Luo-Schleimer-Tillman \cite{LST} states that, up to passing to 
a finite cover, any complete cusped hyperbolic $3$-manifold may be seen as the interior of a compact 
triangulated manifold.

From now on, we fix a triangulation $\mathcal T$ of a compact manifold with boundary $\partial M$. We 
moreover add some combinatorial hypothesis on the triangulation: we assume that the link of any vertex is 
 a disc, a torus or an annulus -- \cite[Section 5.1]{BFG} and \cite[Section 2.1]{DGG}. Thus the boundary 
$\partial M$ decomposes as a union of hexagons lying in the boundary of the complex $K_\mathcal{T}$ and 
discs, tori and annuli lying in the links of the vertices. The latter are naturally triangulated by the traces of the tetrahedra. 

\subsection{Flags, Affine Flags}

As in the work of Fock and Goncharov \cite{FG}, the main technical tool will be the flags, affine flags, and their configuration.

Let $V=\C^n$, with its natural basis $(e_1,\ldots,\, e_n)$. All our flags will be complete: they are defined as "a line in a plane in a $3$-dim plane... in a hyperplane".

 More precisely, consider the exterior powers of $V$ and their projectivizations, for $m=1$ to $n-1$:
$$\Lambda^m V\textrm{ and }\P(\Lambda^m V).$$
Note that $\Lambda^1V\simeq V$ and $\Lambda^{n-1} V\simeq V^*$, the dual of $V$. We fix \emph{once for all} the isomorphism $\Lambda^n(V)\simeq \C$ by assigning $1$ to the element $e_1\wedge\ldots\wedge e_n$.

The space of flags in $V$ is a subset of $\prod_1^{n-1}\P(\Lambda^m V)$. To describe it, recall that $G$ acts on each exterior power of $V$, hence diagonally on the product. Moreover the standard flag $F_{\mathrm{st}}$ is defined by:
$$F_\mathrm{st}=([e_1],[e_1\wedge e_2],\ldots,[e_1\wedge \ldots \wedge e_{n-1}]).$$
Then the flag variety is the orbit of $F_\mathrm{st}$
$$\Fl:=\PGL(n,\C)\cdot F_\mathrm{st}\subset \prod_1^{n-1}\P(\Lambda^m V).$$
As the stabilizer of $F_\mathrm{st}$ is the Borel subgroup $B$ of the upper triangular matrices, we have $\Fl\simeq \PGL(n,\C)/B$.

The affine flag variety $\AFl$ lies above $\Fl$. It is a subset of the product $\prod_1^{n-1} \Lambda^m V$ defined as the orbit under $\SL(n,\C)$ of the standard affine flag
$$F_\textrm{aff.st}=(e_1,e_1\wedge_2,\ldots,e_1\wedge \ldots \wedge e_{n-1}).$$ 
As above, we get an isomorphism $\AFl\simeq \SL(n,\C)/U$, where $U$ is the subgroup of unipotent upper triangular matrices. 

We have a natural projection $\AFl\to \Fl$ consisting in projectivizing each coordinates.

Let us introduce an additional notation: if $F$ is a flag (or affine flag) and $1\leq m\leq n-1$, we denote by $F(k)$ its $k$-th coordinate in $\P(\Lambda^mV)$ (or $\Lambda^m V$).

\subsection{Tetrahedra of affine flags}\label{ss:affine}

Coordinates for a triangle of affine flags may be defined following \cite{FG}.  Consider the $n-1$-triangulation (see \cite[Section 1.16]{FG}) of a triangle $ijk$: that is, suppose your triangle is define in the plane by 
$$x+y+z=n-1, \, x,\, y \textrm{ and }z\textrm{ positive}.$$
And consider the triangulation given by the lines $x=p$ or $y=p$ or $z=p$, for $p=1$ to $n-1$. Each of this line is oriented as the parallel edge of the triangle (see figure \ref{fig:ntriang}).
The crossings of this line are the points with integer and non vanishing coordinates $x$, $y$, $z$ in the triangle. The oriented lines of the triangulation define a set of oriented edges between these crossings.

\begin{figure}
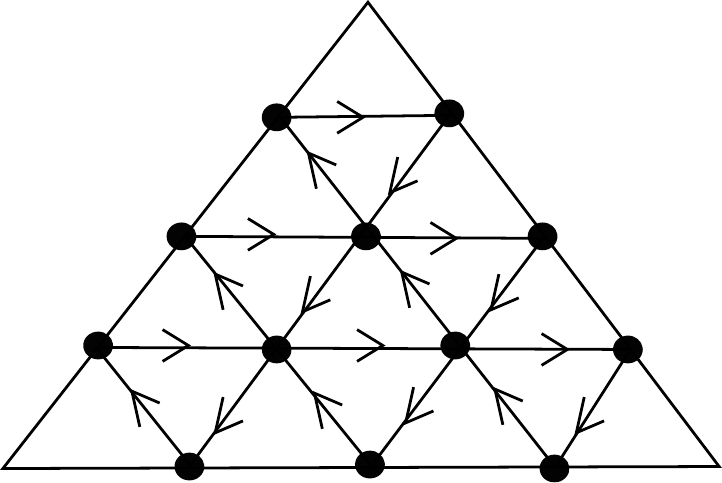
\caption{The $3$-triangulation of a triangle}\label{fig:ntriang}
\end{figure}

For a tetrahedron $T$, we consider the $n-1$-triangulation of its four faces.
As in \cite[Section 4.1.1]{BFG}, let $I_T$ be the set of crossings of the lines. Once again, the oriented lines of the triangulation define a set of oriented edges between neighbor points in $I_T$. We denote by $\alpha$ the elements of $I_T$.

Let $J_T^2\simeq \Z^{I_T}$ be the free $\Z$-module generated by $I_T$ and $(e_\alpha)_{\alpha\in I_T}$ its natural basis. Define a $2$-form $\Omega^2$ by, for $\alpha$, $\beta\in I_T$:
$$\Omega^2(e_\alpha,e_\beta)=\epsilon_{\alpha\beta},$$
where $\epsilon_{\alpha\beta}$ is the number of edges from $\alpha$ to $\beta$ minus the number of edges from $\beta$ to $\alpha$.

Denote by $(J^2)^*=\Hom(J^2,\Z)$ its dual $\Z$-module\footnote{See \cite[Section 4.1.2]{BFG} for a presentation of $\Z$-modules, duality and tensorization.}.
Then, a tetrahedron of affine flags $T=(F_1,F_2,F_3,F_4)$ in general position gives a point in $\C^\times \otimes J_T^2\simeq \Hom(J_T^2,\C^\times)$ by the following rule. Let $\alpha$ be an element of $I_T$. Let $ijk$ be an oriented face containing $\alpha$. Then $\alpha$ can be written as the barycenter of $i$, $j$, and $k$ with nonnegative integer weights $a$, $b$, $c$ verifying $a+b+c=n$. Then define:
$$a_\alpha(T)=F_i(a)\wedge F_j(b)\wedge F_k(c)\in \Lambda^nV\simeq \C.$$
The fact that the flags are in general position ensures that $a_\alpha(T)\in \C^\times$.

But there is a problem if $\alpha$ lies on an edge $ij$ and $n$ is even: whether we consider $\alpha$ to belong to one or the other adjacent face, the relative coordinate $a_\alpha(T)$ may change sign. In order to fix it, we assign to the barycenter of $i$ and $j$ with weight $a\leq b$ the coordinate:
$$a_\alpha(T)=F_i(a)\wedge F_j(b).$$
First, the less weighted coordinate.

Section 8 of \cite{FG} proves that a tetrahedra of affine flags $T$ is determined by the data:
$$a(T)=\sum_{\alpha\in I_T}a_\alpha(T) e_\alpha\in \C^\times \otimes J_T^2.$$

Moreover, consider the new tetrahedron of affine flags $T'$ given by multiplying the vector $F_i(m)$ by 
some $\lambda\in \C^\times$ (for some $1\leq i\leq 4$ and $1\leq m\leq n-1$). Then the vector $a(T')$ is 
related to $a(T)$ by:
$$a(T')=a(T)+\lambda v_i(m)$$
where $v_i(m)$ is the sum of the points of $I_T$ lying on the $m$-th plane parallel to the face $jkl$ 
(counted from the face), see figure \ref{fig:v_i(m)}.
One checks that the set of vectors $v_i(m)$ generates the kernel $\ker(\Omega^2)$.

\begin{figure}
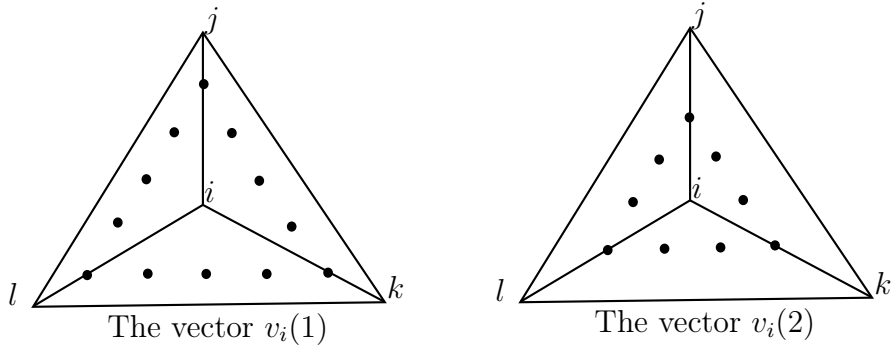
\caption{The vectors $v_i(1)$ and $v_i(2)$ for $n=4$}\label{fig:v_i(m)}
\end{figure}
\subsection{Tetrahedra of flags}

Consider the map:
$$p\colon \: J_T^2 \to (J_T^2)^*$$
given by $p(v)=\Omega^2(\cdot, v)$.
Let $J_T^*$ be its image, and $J_T=J_T^2/\ker(\Omega^2)$ be its dual $\Z$-module.
Then, one checks that this two spaces share the same dimension $2(n-1)^2$.
Let $p^*\colon\: \C^\times \otimes J_T^2 \to \C^\times\otimes J_T^*$ be the dual map.

To a tetrahedron of flags $T$, one associates a point in $\C^\times \otimes J_T^*$ by the following way: 
let $T_\textrm{aff}$ be a lift of $T$ as a tetrahedron of affine flags. And define
$$z(T)=p^*(a(T_\textrm{aff}))\in \C^\times \otimes J_T^*.$$
The considerations at the end of the previous section imply that $z(T)$ is well-defined. This coincide 
(up to a sign) with the $\mathcal X$-coordinates of Fock and Goncharov defined using tri-ratios and 
cross-ratios, and with the coordinates defined in \cite{BFG} for $n=3$.

Note that the space $J_T^*$ carries a natural $2$-form $\Omega^*$ defined by: if $v=p(u)$ and $v'=p(u')$ belong to $J^*_T$, then
$$\Omega^*(v,v')=\Omega^2(u,u').$$ 
This form is symplectic. Similarily, $J_T$ carries a symplectic form $\Omega$ defined as the projection of $\Omega^2$ to $J_T$. The forms $\Omega$ and $\Omega^*$ match through duality.

Dimofte-Gabella-Goncharov \cite{DGG} gave coordinates (called octahedron coordinates) for a tetrahedron 
of flags and relate them to the $\mathcal X$-coordinates of Fock and Goncharov. They proved (see \cite[Section 4]{DGG}) that the subset of $\C^
\times \otimes J^*$ consisting of vectors $z(T)$ associated to an actual tetrahedron of flag form a 
lagrangian submanifold $\mathcal L_T$.

The space $\C^\times\otimes J_T^*$ parametrize the space of framed flat $\PGL(n,\C)$-connections on the boundary 
of the tetrahdron (i.e. a sphere with four holes). Belonging to $\mathcal L_T$ is a fillability condition: 
does the connection extend to the interior of the tetrahedra. In terms of representations of groups, $\C^
\times\otimes J_T^*$ describes the (decorated) representations of the fundamental group of the four-holed 
sphere that are unipotent (the loop around a hole is mapped to a unipotent element). The representations parametrized by $\mathcal L_T$ equal the identity.

For now on, we will mostly forget about the lagrangian sub-manifold $\mathcal L_T$ and work at the level of $J_T^*$.

\subsection{Holonomy in a tetrahedron}

Consider a tetrahedron $T$, and mark three points in each face, one near each vertex. Join the points in 
the same face and at the same vertex. The resulting graph may be realized as the 3d associahedron \cite
[Section 4.3]{DGG}. Then an element $z$ 
in $\C^\times\otimes J^*_T$ defines a holonomy representation, that is a matrix of $\PGL(n,\C)$ 
associated to each oriented edge of the graph. Indeed, from \cite{FG}, such a $z$ parametrize a framed 
flat $\PGL(n,\C)$ 
connection on the four-holed sphere and as such give an holonomy representation. More precisely, each of 
the above mentioned point defines a \emph{snake} and thus a projective basis \cite[Sections 9.7, 9.8]{FG} 
and \cite[Section 4]{DGG}. The matrics are then base changes. The $\PGL(3,\C)$-case may be explained 
without the use of snakes, see \cite[Section 5.4]{BFG}.

\section{Decorated complex and holonomy}

We glue here tetrahedra together, in order to get information on the space of representation of $\pi_1(M)$. There are constraints, the analogous of the gluing equations.

\subsection{Gluing equations}

The gluing equations are the conditions we have to impose in order to glue the tetrahedra. So let $T_1$,$
\ldots$, $T_\nu$ be the tetrahedra of the triangulation of $M$, and $z(T_\mu)$ be their coordinates as 
tetrahedra of flags. Denote by $I$ the vertices of the $I_{T_\mu}$ that remain after gluing in the interior of the complex $K$. These vertices belong to the internal faces and edges of $K$. Each element of $I$ may be seen as a subset of $\cup_{\mu=1}^\nu I_{T_\mu}$. This subset consists of two element if the vertex in $I$ is in a face of the complex $K$ and of several if it is on an edge.

The constraints have been described in \cite{DGG} and are natural generalization of those of \cite{BFG}. 
Indeed, when two faces of $T_\mu$ and $T_{\mu'}$ are glued, one should ensures that the triangle of flags 
decorating them match (up to the orientation). This translates into:\\
{\bf Faces equation:} If the face point $\alpha$ in the tetrahedron $T_\mu$ is glued to the point $
\alpha'$ in $T_\mu'$, then $$z_\alpha(T_\mu)z_{\alpha'}(T_{\mu'})=1.$$

Another condition is that the holonomy of looping around an edge should be equal to the identity. This 
translates (cf explanation for the holonomy below) into:\\
{\bf Edges equations:} For an edge $ij$ of the complex and $a$, $b$ two integer with $a+b=n$, let $T_1$,$
\ldots$, $T_\mu$ be the tetrahedra abutting to the edge $ij$. Then, fix some integer $1\leq m\leq n-1$ 
and denote by $\alpha$ the $m$-th element of $I$ on the edge $ij$ (counting from $j$) in any of the $T_
\gamma$. Then we should impose:
$$\prod_{\gamma=1}^\mu z_\alpha(T_\gamma)=1.$$

\subsection{Holonomy and the decorated variety of representations}\label{ss:holonomy}

Let $J^*$ be the orthogonal sum of the $J^*_{T_\mu}$ and still denote $\Omega^*$ the symplectic form on  
it. Let us construct a graph by considering the associahedra associated to the tetrahedra and adding an edge between any pair of points lying on glued faces near the same vertex \cite[Section 4]{DGG}.

A point $z\in \C^\times\otimes J^*$ represents a set of framed flat $\PGL(n,\C)$-connection on each boundary of the tetrahedra. If it fulfills face and edge equations, this induce a holonomy representation for the graph constructed above. Here is how to compute this representation. First choose a loop in this graph and decompose it into three elementary steps \cite[Section 5.4]{BFG}:
\begin{enumerate}
\item An edge between two vertices of the graph lying on the same face (say the vertices $i$ to $j$ in the face $ijk$ of a tetrahedron $T$).
\item Turning left around an edge $ij$ in a tetrahedron $T$ and landing in the following tetrahedron. That is following the edge from the vertex near $i$ in the face $ijk$ of the tetrahedron $T$ to the vertex near $i$ in the face $ilj$ of the same tetrahedron, and then jump to the vertex near $i$ in the face $ijl$ of the glued tetrahedron.
\item Similarly, turning right around an edge $ij$ in a tetrahedron $T$ and landing in the following tetrahedron.
\end{enumerate}
Then, each of this step corresponds to a base change that can be computed. Indeed, let $T$ be the tetrahedron in which it takes place and let $z=(z_\alpha)_{\alpha\in I_T}\in \C^\times\otimes J_T^*$ be its associated coordinates.
Then there are three matrices $T(z)$, $L_{ij}(z)$ and $R_{ij}(z)$ corresponding to the three base changes.

We are not interested here in describing $T(z)$. In the case $\PGL(3,\C)$ it is given in \cite[Section 
5.4]{BFG} and in the general case may be computed using either \cite[Section 9]{FG} or \cite[Section 4]
{DGG}. From the same references, we compute the matrices $L_{ij}$ and $R_{ij}$. Denote by $ij(m)$ the $m$-th point in $I_T$ lying on the edge (counting from $j$). Then we get that $L_{ij(m)}$ is a diagonal 
matrices depending only on the edge coordinates $z_{ij(m)}$ (see \cite[Lemma 9.3]{FG}):
$$L_{ij}(z)=\begin{pmatrix} 
1 & 0&0&\ldots&0 \\
 0& z_{ij(1)}&0&\ldots&0\\
 0&0&z_{ij(1)}z_{ij(2)}&\ldots &0\\
 \vdots &&&\ddots & \vdots\\
 0&\ldots&\ldots &0&\prod_{m=1}^{n-1}z_{ij(m)}\end{pmatrix}
$$

The computation for $R_{ij}(z)$ is harder. But we are only interested here in its diagonal 
part. In order to describe it, define $Z_{ij(m)}$ to be the product of all $z_\alpha$ for $
\alpha\in I_T$ lying at the level $m$ above the face $jlk$ and \emph{not} in the face $ilk$ 
(see figure \ref{fig:Rij}).
\begin{figure}
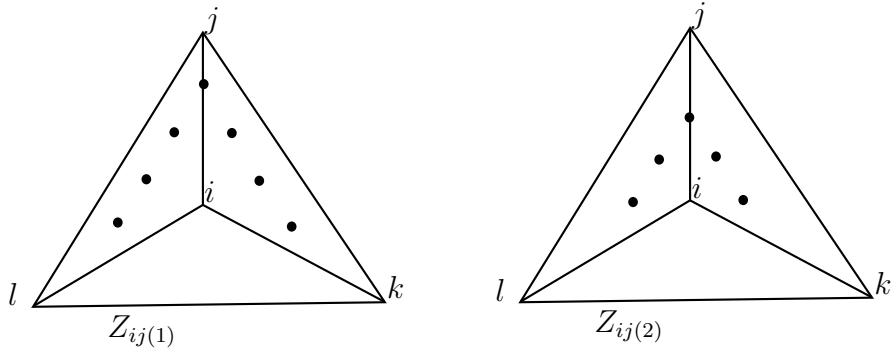
\caption{The points involved in the computation of $Z_{ij(1)}$ and $Z_{ij}(2)$ for $n=4$}\label{fig:Rij}
\end{figure}
Then, from \cite{DGG} or \cite[Sections 9.8 and 9.9]{FG}, one gets:
$$R_{ij}(z)=\begin{pmatrix} 
1 & \star&\star&\ldots&\star \\
 0& \pm (Z_{ij(1)})^{-1}&\star&\ldots&\star\\
 0&0&\pm (Z_{ij(1)}Z_{ij(2)})^{-1}&\ldots &\star\\
 \vdots &&&\ddots & \vdots\\
 0&\ldots&\ldots &0&\pm \prod_{m=1}^{n-1}(Z_{ij(m)})^{-1}\end{pmatrix}
$$

Remark that the fact that a point $z\in \C^\times \otimes J^*$ fulfills the edge and face 
conditions and the lagrangian constraint implies (and in fact is equivalent to) that if two 
loops in the graph based at the same vertex of the graph are homotopic in $M$, then their 
holonomies are equal. This explain why such a $z$ parametrizes (decorated) representations 
of $\pi_1(M)$.

\section{Coordinates for the boundary and the symplectic isomorphism}

Fix once again an element $z\in \C^\times\otimes J^*$, seen as a collection of framed flat $
\PGL(n,\C)$-connections on the tetrahedra boundaries. If it fulfills the face an edge 
conditions, it should induce a framed flat $\PGL(\C)$-connection on $\partial M$. We 
explain here how to describe this connection using coordinates.

Recall that $\partial M$ decomposes as the union of the boundary of the tetrahedra complex 
$K$ and discs, tori or annuli lying in the links of the vertices of $K$. Discs will not 
need coordinates, as the associated moduli space is trivial. We describe first the 
coordinates for the boundary of $K$ and then for the tori/annuli part.

We define $J^2$ as the orthogonal sum of the $J^2_T$, and we keep the notation $\Omega^2$ for its $2$-form.

\subsection{Boundary of the complex}

The boundary of the complex is homeomorphic to a punctured triangulated surface $\Sigma$. 
We use the usual Fock and Goncharov coordinates for this surface \cite[Section 9]{FG}. 
Namely, let $I_\Sigma$ be the vertex of the $n-1$-triangulation of $\Sigma$.  Define $J^2_
\Sigma=\Z^{I_\Sigma}$. This $\Z$-module carries a $2$-form $\Omega_\Sigma$ defined 
similarly to $\Omega^2$ using the oriented edges of the $n-1$-triangulation. Thus there is 
a map:
$$\begin{matrix} p_\Sigma\colon & J^2_\Sigma&\to &(J^2_\Sigma)^* \\ & v &\mapsto & \Omega_\Sigma(\cdot,v)\end{matrix}$$

We denote by $J_\Sigma^*$ its image and $J_\Sigma$ the quotient of $J^2_\Sigma$ by the kernel $\ker(\Omega_\Sigma)$.
Restricting the framed flat $\PGL(n,C)$-connection given by $z$ to the bounday of $K$ yields such connexion on $\Sigma$. Its coordinates belong to $J^*_\Sigma$. This  operation defines a map $J^*\to J^*_\Sigma$ that is so defined in coordinates:
for $\alpha \in I_\Sigma$, there is a subset of $\cup_{\mu=1}^\nu I_{T_\mu}$ consisting of the $\beta$ that are identified to $\alpha$ after gluing. For each of these $\beta$, denote by $z_\beta$ the corresponding coordinate of $z$. Then $z_\alpha^\Sigma$ verifies:
$$z_\alpha^\Sigma\prod_{\beta\textrm{ identified to }\alpha}z_\beta=1.$$

\subsection{Coordinates for tori and annuli}\label{ss:coordinatelinks}

We choose, once for all, a symplectic basis of the homology $(l,m)$ for each tori and a 
generator $s$ of the homology together with a generator $t$ of the homology relative to the 
boundary for each annuli, with intersection number $\iota (s,t)=1$. Each of these tori and 
annuli is the link of a vertex of the tetrahedra complex. We choose for each of 
them a representative as a path as in section \ref{ss:holonomy} which remains near the 
vertex. Denote by $\nu_t$ the number of torus links and $\nu_a$ the number of annulus links.

Using the rules of section \ref{ss:holonomy}, one may compute the holonomy of this paths. 
This is always a product of upper-triangular matrices. Denoting by $\rho$ the holonomy representation associated to $z$, one may write:
$$\rho(l)=\begin{pmatrix} 1&\star &\ldots &\star\\0&L_1&\star&\star\\
\vdots&&\ddots&\vdots\\ 0&\ldots&0&L_1L_2\cdots L_{n-1}\end{pmatrix}$$
and define accordingly the number $(M_m)$, $(S_m)$, $(T_m)$ for $1\leq m\leq n-1$.

The coordinates associates to the boundary are these vectors:
$$(L_m,M_m)\in (\C^\times)^{2(n-1)}\textrm{ for each torus}$$
$$\textrm{ and }(S_m,T_m)\in (\C^\times)^{2(n-1)}\textrm{ for each annulus}.$$
We denote by $(L,M,S,T)\in(\C^\times)^{2(n-1)(\nu_a+\nu_t)}$ this vector.
 
This spaces carry a natural symplectic form, the Goldman-Weil-Peterson form $\mathrm{wp}$. It is formally defined as the coupling of the cup-product and the Killing form on $\mathfrak {sl}(n,\C)$. We will define it precisely later on.
 
The main result of our paper is stated as follows:
\begin{theorem}\label{th:main}
Restricted to the subvariety of $\C^\times \otimes J^*$ defined by the face and edge conditions, the $2$-form $\Omega^*$ is the pull-back by the map 
$$z\mapsto (z^\Sigma,L,M,S,T)$$
of a $2$-form $\mathrm{wp}$ on $(\C^\times\otimes J_\Sigma^*)\times (\C^\times)^{2(n-1)(\nu_a+\nu_t)}$.

Moreover $\mathrm{wp}$ coincide with the Weil-Petersson form in restriction on 
each torus or annulus and with $\Omega_\Sigma$ in restriction to $\C^\times
\otimes J_\Sigma^*$. For this form $\mathrm{wp}$, the tori part is orthogonal 
to the annuli part and the boundary part. However there is a coupling between 
the annuli and boundary parts.
\end{theorem}
The form $\mathrm{wp}$ should be the Weil-Petersson form on the space of representation of $\partial M$. Unfortunaltely, this is not yet clear from the literature.

In order to prove this theorem, we remark that, let alone the lagrangian condition, 
every condition is expressed as "a product of $z$-coordinates$=\pm1$". So this is a good 
idea to linearize everything.

\section{Linearization}

This section is a direct generalization of \cite[Section 7]{BFG}.

\subsection{Face and edge conditions} 
 
We consider another $\Z$-module\footnote{This corresponds to $C_1^\textrm{or}+C_2$ in \cite{BFG}.}: $\Z^I$. Recall that $I$ is the set of vertices of the $n-1$-triangulations of the tetrahedra that remain in the interior of the complex $K$ after gluing. Let $(e_\alpha)_{\alpha\in I}$ be its natural basis. Any $\alpha \in I$ may be seen as a subset of $\cup_{\mu=1}^\nu I_{T_\mu}$. This yields a map:
$$\begin{matrix}F\colon & \Z^I&\to&J^2\\& e_\alpha &\mapsto & \sum_{\beta\in \alpha} e_\beta\end{matrix}$$
By duality, one gets a line of applications:
$$\Z^I\xrightarrow{F}J^2\xrightarrow{p}(J^2)^*\xrightarrow{F^*}(\Z^I)^*.$$
From now on, we identify $\Z^I$ with its dual through the canonical basis.

\begin{lemma}
The composition $F^*\circ p \circ F$ vanishes.
\end{lemma} 
 
 \begin{proof}
 This is an inspection without difficulty. For example, if $\alpha$ is inside a face of the complex $K$, $F(\alpha)$ is of the form $e_\alpha(T_\mu)+e_{\alpha}(T_{\mu'})$. Applying $p$, for each neighbor $\beta$ of $\alpha$, we get a vector $\pm  (e_\beta^*(T_\mu)-e_{\beta}^*(T_{\mu'}))$. Hence, the $e_\beta^*$-component of $F^*\circ p\circ F(e_\alpha)$ vanishes, as well as every other component.
 \end{proof}
 
Following closely \cite[Section 7.3]{BFG}, letting $G : J \rightarrow Z^I$ be the map induced by $F^* \circ p$ and 
$F' : \Z^I \to J$ be the map
$F$ followed by the canonical projection from $J^2$ to $J$, we get a complex:
\begin{equation} \label{complexd}
\Z^I \stackrel{F'}{\rightarrow} J \stackrel{G}{\rightarrow} \Z^I .
\end{equation}
Similarly, letting $G^*= p \circ F$ and $(F')^*$ be the restriction of $F^*$ to $\mathrm{Im} (p) = J^*$ we get the dual complex:
\begin{equation} \label{complex}
\Z^I \stackrel{G^*}{\rightarrow} J^* \stackrel{(F')^*}{\rightarrow} \Z^I .
\end{equation}
We define the homology groups of these two complexes:
$$\mathcal{H} (J) = \mathrm{Ker} (G) / \mathrm{Im} (F') = \mathrm{Ker} (F^* \circ p) / (\mathrm{Im} (F) + \mathrm{Ker} ( p))$$ 
and 
$$\mathcal{H} (J^*) = \mathrm{Ker} ((F')^*) / \mathrm{Im} (G^*) = (\mathrm{Ker}(F^* ) \cap \mathrm{Im} ( p)) / \mathrm{Im} (p \circ F).$$ 
We note that:
\begin{equation} \label{eq:ortho}
\mathrm{Ker} (G) = \mathrm{Im} (F')^{\perp_{\Omega}} \mbox{ and } \mathrm{Im} (G^*) = \mathrm{Ker} ((F')^*)^{\perp_{\Omega^*}}.
\end{equation}
The symplectic forms $\Omega$ and $\Omega^*$ thus induce skew-symmetric bilinear forms on 
$\mathcal{H} (J)$ and $\mathcal{H} (J^*)$. These spaces are obviously dual spaces and the bilinear
forms match through duality.

We claim that $(F')^*$ linearize the face and edge equations:
\begin{lemma}
An element $z\in\C^\times\otimes J^*$ fulfills the face and edge equations if and only if: 
$$z\in \C^\times\otimes\ker((F')^*).$$
\end{lemma}
\begin{proof}
Once again this is proved by inspection: the $e_\alpha^*$ component of 
$(F')^*(z)=z\circ F'$ 
is the product of the component $z_\beta$ for $\beta$ belonging to $\alpha$. If $\alpha$ 
sits on a face, this gives a face condition; if $\alpha$ sits on an edge, this gives an 
edge condition. 
\end{proof}

\subsection{Coordinates for the links}

The coordinates we have constructed for a torus $T$ may be seen as an element of 
$H^1(T,(\C^\times)^{n-1})$. We construct now a map at the level of the chains. Once again, we are very close of \cite[Section 7.1]{BFG}.

\subsubsection{Simplicial decompositions of the links}

Each boundary surface $S$ in the link of a vertex is triangulated by the traces of the 
tetrahedra; from this we build the CW-complex $\mathcal D$ whose edges consist of the inner 
edges of the first barycentric subdivision, see Figure \ref{fig:celldecomposition}. We 
denote by $\mathcal D'$ the dual cell division. Let $C_1 (\mathcal D) = C_{1} (\mathcal D ,
\Z)$ and $C_1 (\mathcal D') = C_{1} (\mathcal D' , \Z)$ be the corresponding chain 
groups. Given two chains $c \in C_1 (\mathcal D)$ and $c' \in C_{1} (\mathcal D' )$ we 
denote by $\iota (c, c')$ the (integer) intersection number of $c$ and $c'$. This defines a 
bilinear form $\iota : C_1 (\mathcal D ) \times C_1 (\mathcal D' ) \to \Z $ which induces 
the usual intersection form on $H_1 (S)$ (or between the homology and the homology relative to the boundary in the annulus case). In that
way $C_1 (\mathcal D' )$ is canonically isomorphic to the dual of $C_1 (\mathcal D)$. 

\begin{figure}[ht]      
\begin{center}
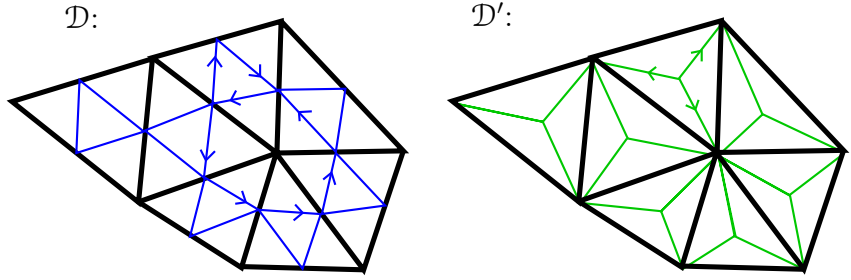
\caption{The two cell decompositions of the link} \label{fig:celldecomposition}
\end{center}
\end{figure}

\subsubsection{Goldman-Weil-Petersson form for tori} \label{ss:omegatori} Here we equip 
$$C_1 (\mathcal D , \R^{n-1}) = C_1 (\mathcal D ) \otimes \R^{n-1}$$
with the bilinear form $\omega$ defined by coupling the intersection form $\iota$ with the scalar product on $\R^{n-1}$ seen as the space of roots of $\mathfrak{sl}(n,\C)$ with its Killing form. We describe more precisely an integral version of this.

From now on we identify $\R^{n-1}$ with the subspace $V = \{ (x_m)_{1\leq m\leq n} \in \R^n \; : \; \sum_m x_m=0\}$ via 
the map sending the $m$-th vector $v_m$ of the canonical basis to $(0,\ldots,0,1,-1,0,\ldots,0)^t$, the entry $1$ being the $m$-th.

We let $L \subset V$ be the standard lattice in $V$ where all the coordinates are in $\Z$. We identify it with $\Z^{n-1}$ using the above basis of $V$.   
The restriction of the usual euclidean product of $\R^n$ gives a product, denoted $[,]$, on $V$ (the ``Killing form'')\footnote{In terms of roots of $\mathfrak{sl}(n)$, the choosen basis is the usual basis of positive simple roots.}. In other words, the matrix of the scalar product is the Cartan matrix: all entries are $0$, except the diagonal which is filled with $2$ and the upper and lower-diagonals, filled with $-1$.

Identifying $V$ with $V^*$ using the scalar product $[,]$, the dual lattice $L^* \subset V^*$ becomes a lattice $L'$ in $V$; an element $y \in V$ belongs to $L'$
if and only if $[x,y] \in \Z$ for every $x \in L$. 

We consider $C_1 (\mathcal D , L)$ and define $\omega = \iota \otimes [ \cdot , \cdot ] : C_1 (\mathcal D , L) \times C_1 (\mathcal D' , L') \rightarrow \Z$ by the formula
$$\omega \left( c \otimes l , c' \otimes l' \right) = \iota (c,c') \left[ l , l' \right].$$
This induces a (symplectic) bilinear form on $H_1 (S, \R^2)$\footnote{Or a coupling between the homology and the homology relative to the boundary in the annulus case.} which we still denote by $\omega$. 
Note that $\omega$ identifies $C_1 (\mathcal D' , L')$ with the dual of $C_1 (\mathcal D , L)$. 

\begin{remark}
The canonical coupling $C_1 (\mathcal D , L) \times C^1 (\mathcal D , L^*) \rightarrow \Z$ identifies $C_1 (\mathcal D , L)^*$ with $C^1 (\mathcal D , L^*)$. This 
last space is naturally equipped with the ``Goldman-Weil-Petersson'' form $\mathrm{wp}$, dual to $\omega$. Let $\langle , \rangle$ be the natural
scalar product on $V^*$ dual to $[,]$: letting $d:V\to V^*$ be the map defined by $d(v) = [v, \cdot ]$, we have
$\langle d (v) , d (v') \rangle = [v,v']$. 
On $H^1 (S , \R^2)$ the bilinear form $\textrm{wp}$ induces a symplectic form --- the usual Goldman-Weil-Petersson symplectic form --- formally defined as the coupling of the cup-product and the scalar product $\langle , \rangle$.
\end{remark}

\subsection{Peripheral holonomy}

To any decoration $z\in \C^{\times}\otimes (J^*\cap \mathrm{Ker}(F^*))$ we now explain how to associate an element 
$$\textstyle \Hol_\periph(z) \in \mathrm{Hom} (H_1(S , L), \C^{\times} ).$$
We may represent any class in $H_1 (S, L)$ by an element $c \otimes (x_1,\ldots,x_{n-1})^t$ in $C_1 (\mathcal D , L)$ where $c$ is a closed path in $S$ (seen as the link of the corresponding vertex in the complex $K$). 
Using the decoration $z$ we may compute the holonomy of the loop $c$, as explained in Section~\ref{ss:holonomy} (see also section \ref{ss:coordinatelinks}): it is  an upper triangular matrix. Let us write the diagonal part: $$(1,C_1,\ldots,\prod_1^{n-1}C_m).$$
 The application which maps 
$c \otimes (x_1,\ldots,x_{n-1})^t$ to $\prod_1^{n-1} C_m^{x_m}$ is the announced element $\Hol_\periph(z)$ of $\C^{\times}\otimes H^1 (S , L^* )$. 

In the case of an annulus, we obtain in the same way a map, still denoted $\Hol_\periph$, from $\C^\times\otimes (J^*)$ to $\C^{\times}\otimes (H^1 (S , L^* )\times H^1(S,\partial S,L^*)$.

The choice of a given longitude and meridian gives a basis of $H_1(S)$. It allows to identify $\C^{\times}\otimes H^1 (S , L^* )$ with $(\C^\times)^{2(n-1)}$. This explains our definition of coordinates in section \ref{ss:coordinatelinks}.

\subsubsection{Linearization of the holonomy elements} \label{ss:linearizationholonomy}

We now linearize the map $\Hol_\periph$, i.e. we explain how the computations of the eigenvalues of the holonomy of the torus may be done in our framework of $\Z$-modules.

We define the linear map $h: C_1 (\mathcal D , L) \to J^2$ on a basis. Let 
$c_{ij}^\mu$ be the edge turning left around the edge $(ij)$ in the tetrahedron $T_{\mu} = (ijkl)$, see Figure \ref{fig:h}. For $1\leq m \leq n-1$, let $c_{ij}^\mu(m)$ be the tensor of $c_{ij}^\mu$ with the $m$-th canonical basis vector $v_m$ of $L\simeq \Z^{n-1}$. Parametrize the points of $I_{T_\mu}$ in the two faces containing the edge $ij$: $\alpha^m_l$ (for $1\leq m\leq n-1$ and $0\leq l\leq n-m)$ is the point at the level $m$ from the face $jlk$ at the position $l$ (counted algebraically and rightward from the point on the edge $ij$) -- see figure \ref{fig:h}. Then we define:
\begin{equation} \label{eq:maph}
h \left( c_{ij}^\mu(m) \right) = 2\sum_{2l<n-m} e_{\alpha^m_l}+\sum_{2l=\pm (n-m)}e_{\alpha^m_l} .
\end{equation}
Remark that the second sum is empty for $n-k$ odd.
 
\begin{figure}[ht]      
\begin{center}
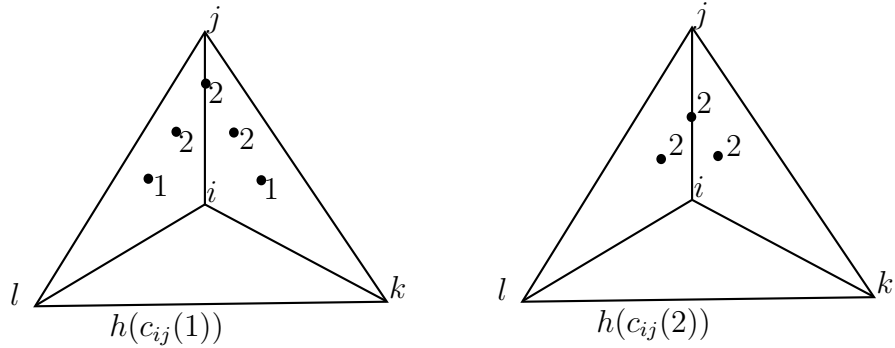
\caption{The map $h$} \label{fig:h}
\end{center}
\end{figure}

The claim that $h$ linearizes $\Hol_\periph$ is given by the following:
\begin{lemma}\label{lem:Holperiph}
Let $z \in k^{\times}\otimes (J^*\cap \mathrm{Ker}(F^*))$. Seeing $z$ as an element of $\mathrm{Hom} (J^2 , C^{\times})$,
we have:
$$z \circ h = \Hol_\periph(z)^2.$$
\end{lemma}

The proof of this lemma goes along exactly the same lines as \cite[Lemma 7.2.1]{BFG}. It is 
a lengthy inspection, whose major difficulty is to define reasonable notations. We postpone 
it until the last section.

Let $h^* : (J^2)^* \to C_1 (\mathcal D , L)^*$ be the map dual to $h$. Note that for any $e \in J^2$ and $c \in C_1(\mathcal D,L)$ we have
\begin{equation} \label{h*}
(h^* \circ p (e))( c) = p(e) (h( c)) = \Omega^2 (e , h( c)).
\end{equation}
Now composing $p$ with $h^*$ and identifying $C_1 (\mathcal D , L)^*$ with $C_1 ( \mathcal D' , L')$ using $\omega$ we get a map
\begin{eqnarray} \label{eq:mapg}
g : J^2 \rightarrow C_1 (\mathcal D' , L')
\end{eqnarray}
and it follows from equation \eqref{h*} that for any $e \in J^2$ and $c \in C_1(\mathcal D,L)$ we have
\begin{equation} \label{g}
\omega (c, g (e)) = \Omega^2 (e , h(c )).
\end{equation}

In the following we let $C_1 (\partial M , L)$ and $C_1 (\partial M ' , L')$
be the orthogonal sum of the $C_1 (\mathcal D , L)$'s and $C_1 (\mathcal D' , L')$'s for each torus or annulus link $S$. We abusively denote by $h : C_1(\partial M) \to J^2$
and $g: J^2 \to C_1 (\partial M' , L')$ the product of the maps defined above on each $T$.

\section{Homologies and symplectic isomorphism in the closed case}

Le us first assume that $K$ is a closed complex. In that case all links are tori. We will come back latter on the general case.

\subsection{Homologies in the closed case}

We defined the homology groups $\mathcal H(J)$ and $\mathcal H(J^*)$ and the chain groups of 
the simplicial decomposition. We claim here that $h$ induces well defined map in homology. This will allow to state our main technical theorem in the closed case.

Let $Z_1 (\mathcal D , L)$ and $B_1 (\mathcal D, L)$ be the subspaces of cycles and boundaries in $C_1 (\mathcal D, L)$. The following lemma is easily checked by inspection.

\begin{lemma}
We have: 
$$h( Z_1 (\mathcal D, L)) \subset \mathrm{Ker} (F^* \circ p)$$
and 
$$h(B_1 (\mathcal D, L )) \subset \mathrm{Ker} (p) + \mathrm{Im} (F).$$
\end{lemma}
In particular $h$ induces a map $\bar{h} : H_1 (\mathcal D , L) \to \mathcal{H} (J)$ in homology. By duality, the map $g$ induces a map 
$\bar{g} : \mathcal{H} (J) \to H_1 (\mathcal D' , L')$ as follows from:
\begin{lemma}
We have:
$$g(\mathrm{Ker}(F^* \circ p))\subset Z_1(\mathcal D',L'),$$
and
$$g(\mathrm{Ker} (p) + \mathrm{Im} (F))\subset B_1(\mathcal D',L').$$
\end{lemma}
\begin{proof}
First of all, $Z_1(\mathcal D',L')$ is the orthogonal of $B_1(\mathcal D,L)$ for the coupling $\omega$. Moreover, by definition of $g$, if $e\in\textrm{Ker} (F^* \circ p)$, we have: 
\begin{eqnarray*}
g(e)\in Z_1(\mathcal D',L') & \Leftrightarrow & \omega(B_1(\mathcal D,L),g(e))=0\\
& \Leftrightarrow & \Omega^2(h(B_1(\mathcal D,L)),e)=0.
\end{eqnarray*}
The last condition is given by the previous lemma. The second point is similar.
\end{proof}
Note that $H_1 (\mathcal D , L)$ and $H_1 (\mathcal D' , L')$ are canonically isomorphic so that we identified them (to $H_1 (\partial M , L)$) 
in the following. 
\begin{theorem} \label{thm:homologies}
\begin{enumerate}
\item The map $\bar{g} \circ \bar{h} : H_1 (\partial M , L ) \to H_1 (\partial M , L)$ is multiplication by~$4$.
\item Given $e \in \mathcal{H} (J)$ and $c \in H_1 (\partial M, L)$, we have 
$$\omega (c, \bar{g} (e)) = \Omega (e , \bar{h} ( c)).$$
\end{enumerate}
\end{theorem}

As a corollary, one understands the homology of the various complexes.

\begin{cor} \label{corhom}
The map $\bar h$ induces an isomorphism from $H_1 (\partial M, L)  $ to $\mathcal{H} (J)$. Moreover we have $\bar h^* \Omega = - 4 \omega$.
\end{cor}

\begin{cor}\label{corkab}
 The form $\Omega^*$ on $\C^\times \otimes (J^*\cap \mathrm{Ker}(F^*))$ is the pullback of $\mathrm{wp}$ on $H^1(\partial M,(\C^\times)^{n-1})$ by the map $\Hol_\periph$. 
\end{cor}

Theorem \ref{th:main} in the closed case is exactly corollary \ref{corkab}.

The proofs of the corollaries from the theorem is given in \cite[Section 7.4]{BFG}. You just have to adapt the dimension of $\mathcal H(J)$ and $\mathcal H
(J^*)$: for $\PGL(n,\C)$, it is $2(n-1)l$ ($l$ is the number of tori links).

\subsection{Proof of theorem \ref{thm:homologies}}\label{ss:proofhomologies}

We want to compute $g\circ h$. Even if it seems simple, this is a point where a new 
approach was needed. In \cite{BFG}, this was dealt with a direct computation, but did not 
show how to generalize it.

\begin{lemma}

\end{lemma}

Let us first work in a single tetrahedron $T=ijkl$. We denote by $c_{ij}$ the edge of $
\mathcal D$ corresponding to a (left) turn around the edge
$ij$ and we denote by $c_{ij}'$ its dual edge in $\mathcal D'$, see Figure \ref{fig:celldecomposition}. We use the notation $c_{ij}(m)$ to denote the tensor product of $c_{ij}$ with the $m$-th canonical basis vector (still denoted $v_m$) of $L$.

\begin{lemma}
Then, for any vector $v\in L$, there exists a vector $v'$ in $L$ such that the image $g\circ h (c_{ij}\otimes v)$ decomposes as:
$$g\circ h (c_{ij}\otimes v)=2 (c_{ik}' - c_{il}') \otimes v +2 (c_{ki}' - c_{kj}' + c_{jl}' - c_{jk} ' + c_{lj}' - c_{li} ' ) \otimes v'.$$
\end{lemma}
\begin{proof}
In view of equation \ref{g}, we need to compute $\Omega^2$ on the image of $h$.

The first computation is straightforward:
$$\Omega^2\left(h(c_{ij}(m)),h(c_{ik}(m'))\right)=2[v_m,v_{m'}].$$
This bracket equals $2$ if $m=m'$, $-1$ if $|m-m'|=1$, and $0$ in the other cases.

Consider now the permutation $\sigma$: $ijkl\to ijlk$. It reverses the orientation of 
the tetrahedron, thus changes $\Omega^2$ into $-\Omega^2$. And it fixes $h(c_{ij}(m)$ and $h(c_{ji}(m))$, while exchanging $h(c_{ik}(m))$ and $h(c_{il}(m))$.

We deduce that:
$$\Omega^2\left(h(c_{ij}(m)),h(c_{il}(m'))\right)=-2[v_m,v_{m'}],$$
$$\Omega^2\left(h(c_{ij}(m)),h(c_{ji}(m'))\right)=0.$$

Moreover, $h(c_{ij}(m))$ has non vanishing components only on the two faces containing $ij$. Thus:
$$\Omega^2\left(h(c_{ij}(m)),h(c_{kl}(m'))\right)=\Omega^2\left(h(c_{ij}(m)),h(c_{lk}(m'))\right)=0.$$

We also claim:
$$\Omega^2\left(h(c_{ij}(m)),h(c_{ki}(m'))\right)=-\Omega^2\left(h(c_{ij}(m)),h(c_{kj}(m'))\right)$$
$$\Omega^2\left(h(c_{ij}(m)),h(c_{jk}(m'))\right)=-\Omega^2\left(h(c_{ij}(m)),h(c_{jl}(m'))\right)$$
$$\Omega^2\left(h(c_{ij}(m)),h(c_{li}(m'))\right)=-\Omega^2\left(h(c_{ij}(m)),h(c_{lj}(m'))\right)$$
Indeed, let us prove the first assertion: we have 
$$h(c_{ki}(m'))+h(c_{kj}(m'))+h(c_{kl}(m'))=v_k(m'),$$
where $v_k(m')$ is a vector generating $\ker(\Omega^2)$ defined in section \ref{ss:affine}. 
Hence 
$$\Omega^2\left(h(c_{ij}(m)),h(c_{ki}(m'))+h(c_{kj}(m'))+h(c_{kl}(m'))\right)=0.$$
We have seen that the last vector of the sum is orthogonal to $h(c_{ij}(m))$. It yields the desired assertion.

We now claim that:
$$
\begin{matrix}\Omega^2\left(h(c_{ij}(m)),h(c_{ki}(m'))\right)& =&\Omega^2\left(h(c_{ij}(m)),h(c_{jl}(m'))\right)\\ &=&\Omega^2\left(h(c_{ij}(m)),h(c_{lj}(m'))\right)
\\&=&-\Omega^2\left(h(c_{ij}(m)),h(c_{kj}(m'))\right) \\&=&-\Omega^2\left(h(c_{ij}(m)),h(c_{jk}(m'))\right) \\&=&-\Omega^2\left(h(c_{ij}(m)),h(c_{li}(m'))\right).
\end{matrix}
$$  
Indeed, using the permutation $\sigma$, we have:
$$\Omega^2\left(h(c_{ij}(m)),h(c_{ki}(m'))\right)=-\Omega^2\left(h(c_{ij}(m)),h(c_{li}(m'))\right).$$
From the previous claim, we get:
$$\Omega^2\left(h(c_{ij}(m)),h(c_{ki}(m'))\right)=\Omega^2\left(h(c_{ij}(m)),h(c_{lj}(m'))\right).$$
We may also write (still with the previous claim:
$$\Omega^2\left(h(c_{ij}(m)),h(c_{jl}(m'))\right)=-\Omega^2\left(h(c_{il}(m)),h(c_{jl}(m'))\right).$$
Applying now the transposition $(jl)$, we get the desired:
$$\Omega^2\left(h(c_{ij}(m)),h(c_{jl}(m'))\right)=\Omega^2\left(h(c_{ij}(m)),h(c_{lj}(m'))\right).$$
This proves the last claim. And the lemma follows, using equation \eqref{g}.
\end{proof}

Let us now glue the tetrahedra: consider a cycle $c=\sum_\mu c_{ij}^\mu$. We let the index $\mu$ be implicit inthe following formulas and compute:
\begin{multline*}
 g\circ h\left( c\otimes v\right)=\left(2\sum_\mu c_{ik}' - c_{il}'\right)\otimes   v+\\ \left(2 \sum_\mu c_{ki}' - c_{kj}' + c_{jl}' - c_{jk} ' + c_{lj}' - c_{li}'\right)\otimes v'
\end{multline*}
The miracle already observed by Neumann \cite[Lemma 4.3]{Neumann} and \cite[Section 8.1]{BFG} happens once again: we are now reduced to a problem in the homology of $\partial M$ and the lattice $L$ does not play any role here. 

We borrow from the last mentioned reference the following lemma \cite[Lemma 8.1.1]{BFG}, which concludes the proof of the theorem (in the closed case):
\begin{lemma}
\begin{itemize}
  \item The path $\sum_\mu c_{ik}' - c_{il}'$ is homologous to $2c$ in $H^1(\partial M)$,
  \item The path $ \sum_\mu c_{ki}' - c_{kj}' + c_{jl}' - c_{jk} ' + c_{lj}' - c_{li}'$ vanishes in $H^1(\partial M)$.
 \end{itemize}
\end{lemma} 

\section{Extension to the general case}

We refer to \cite[Section 9]{BFG} to the extension to the general case. 
Indeed, we have choosen notations coherent with the ones used in \cite{BFG}, 
and the section 9 may be transposed almost \emph{verbatim}. There are only two 
points to adapt: our $\Z^I$ is denoted there $C_1^{\textrm{or}}+C_2$ and in 
the lemma 9.2.2, the computed dimension should be replaced with:
$$\textrm{dim}(\mathcal H(J))=2(n-1)\nu_t+(n-1)\nu_a+\textrm{dim}(J_\Sigma).$$

\section{Local rigidity}

Let us mention that the generalization done in this paper allows also to 
generalize the result of \cite{BFGKR}. This gives a combinatorial proof for all cases of the theorem of Menal-Ferrer and Porti \cite{MenalFerrer-Porti}. Let us state this theorem.

Consider a compact orientable $3$-manifold $M$ with boundary a union of
$\ell$ tori. Let us assume that the interior of $M$ carries a hyperbolic
metric of finite volume and let $\rho_{} : \pi_1 (M) \rightarrow
\mathrm{PGL}(n, \C)$ be the corresponding holonomy representation composed with the
$n$-dimensional irreducible representation of $\mathrm{PGL} (2 , \C)$ (this 
representation is usually called geometric). Denote by $\mathcal R(M,\mathcal 
T)$ the sub-variety of $\C^\times \otimes (J^*\cap\ker (F^*))$ defined by 
asking that the $z$-coordinates on each tetrahedron $T\in\mathcal T$ belong to the langrangian 
manifold $\mathcal L_T$. The variety $\mathcal R(M,\mathcal T)$ comes with a 
natural projection (given by computing the holonomy):
$$\mathcal R(M,\mathcal T)\to \chi_n(M).$$
This projection gives local charts on $\chi_n(M)$.

\begin{theorem} \label{T1}
The class $[\rho_{}]$ of $\rho$ in $\chi_n(M)$ is a smooth point with local dimension $(n-1)\ell$.

Moreover, for an arbitrary choice of a non-trivial curve $c_i$ in each torus boundary ($1\leq i\leq \ell$), the map:
$$\begin{matrix} \mathcal R(M,\mathcal T)&\to&(\C^\times)^{(n-1)\ell}\\
z&\mapsto &(\Hol_\periph(c_i\otimes v_m))_{1\leq i\leq \ell\textrm{ and }1\leq m\leq n-1}\end{matrix}$$
restricts to a local isomorphism in a neighborhood of a point $z$ lifting $[\rho]$.
\end{theorem}

\begin{proof}
The proof of the theorem is exactly the same as in \cite{BFGKR}. It is enough 
(for dimensions reasons, see \cite[Section 6.4]{BFGKR})  to prove that the subvariety $\prod_{T\in \mathcal T}
\mathcal L_T\subset J^*$ is transverse to $\textrm{im}(p\circ F)$. Indeed, the latter 
is the kernel of the differential of $\Hol_\periph$. But we note, as in \cite{BFGKR}, that it lies in the kernel of $\Omega^*$ restricted to $\ker(F^*)$. For $z$ a point in $\prod_{T\in \mathcal T}\mathcal L_T$, it is enough to find a tangent vector such that $\Omega^*(\xi,\bar\xi)\neq 0$ \cite[Lemma 6.3]{BFGKR}.

Now, fix a point $z\in \prod \mathcal L_T$ such that each tetrahedron is in 
fact hyperbolic (see \cite[Section 7.2.1]{DGG}). Through $z$ passes the (non 
trivial) subvariety of points in $\prod \mathcal L_T$ such that each 
tetrahedron is hyperbolic. Consider a tangent vector $\xi$ at $z$ 
which is also tangent to this subvariety. Then, from \cite[Section 7.2.1]{DGG}, one easily gets that 
$$\Omega^*(\xi,\bar \xi)=\frac 16n(n^2-1) \Omega_\textrm{NZ}(\xi,\bar \xi).$$
Here, $\Omega_{\textrm{NZ}}$ denotes the usual Neumann-Zagier form in the hyperbolic case. But, as observed by Choi \cite{Choi}, we have $\Omega_\textrm{NZ}(\xi,\bar \xi)>0$.
\end{proof}

\section{Computation of the peripheral holonomy}

We prove here lemma \ref{lem:Holperiph}:
Let $z \in k^{\times}\otimes (J^*\cap \mathrm{Ker}(F^*))$. Seeing $z$ as an element of $\mathrm{Hom} (J^2 , C^{\times})$,
we have:
$$z \circ h = \Hol_\periph(z)^2.$$

The proof is very similar to \cite[Lemma 7.2.1]{BFG}.

\begin{proof}
Consider a loop $c$ in the link of a vertex of the complex $K$, and write it as a cycle:
$$c=\sum c_{ij}^\mu - \sum c_{ij}^{\mu'}$$
(with the same notations as in section \ref{ss:proofhomologies}).
That is: $c$ turns left around the edges $e_{ij}^\mu$ and right around $e_{ij}^{\mu'}$.

Using the matrix $L$ and $R$ defined in section \ref{ss:holonomy}, we see that 
the diagonal part of the holonomy of $c$ is 
$$\textrm{Diag}(1,C_1,\ldots,\prod_1^{n-1}C_m),$$
where, for $1\leq m\leq n-1$:
$$C_m=\pm\frac{\prod z_{ij}{^\mu}(m)}{\prod Z_{ij}^{\mu'}(m)}.$$
In terms of the map $\Hol_\periph$, one writes:
 $$C_m=\Hol_\periph(z)(c\otimes v_m).$$

We may simplify the formula for $C_m$, as in the proof of 
\cite[Lemma 7.2.1]{BFG}. Thanks to the face equations, if $c$ turns right in 
two following tetrahedra $T_{\mu'}$ and $T_{\eta'}$, then in the product $Z_{ij}^{\mu'}(m)Z_{ij}(m)^{\eta'}$, all the coordinates $z_\alpha$ corresponding to 
the internal face cancel. So in the formula for $C_m$, at the end, it only 
appears face coordinates for the faces at which $c$ change direction.

Let $\mathcal F$ be the set of such faces (seen as a face of the tetrahedra in 
which $c$ turns right). Denote by $z_{\mathcal F}(m)$ the product of every 
$z_{\alpha}$ for $\alpha$ a point in the interior of the face $\mathcal F=ijk$ 
at 
the level $m$ from the base $jk$. We may then rewrite:
$$C_m=\pm\frac{\prod z_{ij}^\mu(m)}{\prod z_{ij}^{\mu'}(m)\prod_{\mathcal F}z_
{\mathcal F}}.$$

Now, what about $z\circ h(c\otimes v_m)$ ? From the definition of $h$ -- see 
equation \eqref{eq:maph}--, with the same notations, we get:
$$z\circ h(c\otimes v_m)=\frac{\prod_{2l<n-m} (z^\mu_{\alpha^m_l})^2+\prod_{2l=\pm (n-m)}z^\mu_{\alpha^m_l}}{\prod_{2l<n-m} (z^{\mu'}_{\alpha^m_l})^2+\prod_{2l=\pm (n-m)}z^{\mu'}_{\alpha^m_l}}.$$
As before, all the face coordinates outside of the faces in $\mathcal F$ simplify. And by inspecting the situation in faces of $\mathcal F$, we get:
$$z\circ h(c\otimes v_m)=\left(\frac{\prod z_{ij}^\mu}{\prod z_{ij}^{\mu'}\prod_{\mathcal F}z_\mathcal F}\right)^2.$$

This proves the lemma.

\end{proof}

\bibliographystyle{amsalpha}
\bibliography{bibli}

\end{document}